\documentclass[11pt,reqno]{amsart}

\usepackage{floatrow}

\usepackage{adjustbox} 

\usepackage{amsmath}
\usepackage{amssymb}
\usepackage[left=1in,top=1in,right=1in,bottom=1in]{geometry}
\usepackage{epsfig}
\usepackage[normalem]{ulem}
\usepackage[usenames,dvipsnames]{color}
\usepackage{todonotes}
\usepackage[colorlinks, citecolor=black, linkcolor=black, urlcolor=black]{hyperref}
\usepackage{bm}
\usepackage{algorithm}
\usepackage{algpseudocode}

\usepackage{enumitem}

\allowdisplaybreaks

\usepackage{hyperref}
\usepackage{cleveref}
\crefname{subsection}{subsection}{subsections}

\usepackage{tikz}
\usetikzlibrary{backgrounds}
\usetikzlibrary{intersections}

\usepackage{booktabs}
\usepackage{array}
\newcolumntype{N}{>{\centering\arraybackslash}m{1in}}
\newcolumntype{G}{>{\centering\arraybackslash}m{2in}}

\numberwithin{figure}{section}

\newtheorem{theorem}{Theorem}[section]
\newtheorem{lemma}[theorem]{Lemma}

\newcommand\eps{\varepsilon}
\newcommand{\E}{\mathbb E}

\newcommand{\Var}{\mathbb V\textrm{ar}}
\newcommand{\Prob}{\mathbb{P}}

\newcommand{\Nn}{{\mathbb N}}

\newcommand{\scr}{\mathcal}
\newcommand{\mb}{\mathbb}

\theoremstyle{definition}
\newtheorem{remark}[theorem]{Remark}

\def\R{{\mathcal R}}

\newcommand{\bigo}{\mathcal{O}}

\title[Hopping Forcing Number in Random $d$-regular Graphs]{Hopping Forcing Number in Random $d$-regular Graphs}

\author{Pawe\l{} Pra\l{}at}
\address{Department of Mathematics, Toronto Metropolitan University, Toronto, Canada}
\email{pralat@torontomu.ca}

\author{Harjas Singh}
\address{Department of Mathematics, Toronto Metropolitan University, Toronto, Canada}
\email{harjas.singh@torontomu.ca}

\date{}

\begin{document}
	
	\maketitle
	
	\begin{abstract}
		Hopping forcing is a single player combinatorial game in which the player is presented a graph on $n$ vertices, some of which are initially blue with the remaining vertices being white. In each round $t$, a blue vertex $v$ with all neighbours blue may hop and colour a white vertex blue in the second neighbourhood, provided that $v$ has not performed a hop in the previous $t-1$ rounds. The objective of the game is to eventually colour every vertex blue by repeatedly applying the hopping forcing rule. Subsequently, for a given graph $G$, the hopping forcing number is the minimum number of initial blue vertices that are required to achieve the objective.
		
		In this paper, we study the hopping forcing number for random $d$-regular graphs. Specifically, we aim to derive asymptotic upper and lower bounds for the hopping forcing number for various values of $d \geq 2$. 
	\end{abstract}

	\section{Introduction and Main Results}

	\subsection{Definitions}\label{sec:definitions} 
	
	Zero forcing was introduced in~\cite{zeroForcing2008}, with the purpose of finding bounds on the maximum nullity of a family of matrices associated with any graph. The game starts with a graph $G$ with some vertices coloured \textbf{blue} and the remaining \textbf{white}. An interpretation of this is that blue vertices contain information, whereas the white vertices are devoid of it. The goal is to use repeated applications of a \textbf{colour change rule} with the objective being to (eventually) turn every vertex blue. It quickly gained popularity as a model to study the spread of information on a given graph, with applications being found in several fields including physics and engineering; see, for example~\cite{brueni_heath_2005, burgarth_giovannetti_2007, severini_2008}, survey on the subject~\cite{fallat2014minimum}, and a recent article published in the Notices of the AMS~\cite{hicks71many}.
	
	The standard colour change rule, denoted by $\scr{Z}$, allows a blue vertex $b$ to \textbf{force} a white vertex $w$ to become blue provided that $w$ is the unique white neighbour of $b$. The wide interest in zero forcing has generated a large volume of work in the last few years, including analyzing it for random graphs~\cite{bal2021zero}, analyzing the probabilistic counterpart for deterministic graphs~\cite{behague2020tight} as well as random ones~\cite{english2021probabilistic}. On the other hand, the \textbf{hopping forcing rule} $\scr{H}$, first added to the set of existing colour change rules in~\cite{barioli_et_al_2012}, allows a blue vertex $b$ to force a white vertex $w$ in the second neighbourhood of $b$ to become blue provided that $b$ has not performed a force yet and each vertex in the first neighbourhood of $b$ is already blue. For a given graph $G$, the \textbf{hopping forcing number} $H(G)$ denotes the cardinality of the set with the minimum number of initial blue vertices, such that all the vertices of $G$ can eventually be coloured blue.
 
	We borrow the notation used in~\cite{carlson2022hopping} and call a blue vertex $b$ at \textbf{step} $t$ \textbf{dormant} if it has neither performed a force yet nor is able to (the reason being that either some neighbour of $b$ is still white or all vertices in the second neighbourhood of $b$ are already blue), \textbf{active} if it has not yet performed a force but is able to, or \textbf{extinct} if it has already performed a force. 
	
	Furthermore, it is important to note that although several vertices can hop simultaneously at step $t$, for our purpose we only consider \textbf{sequential} hopping, that is, at any given step there will be exactly one vertex performing a force, and if no vertex can hop but there are still some vertices that are white, then the process terminates unsuccessfully. 

\medskip
  
	Now, we set up our set notation for this process and  carefully lay out the steps. Let $B_1$ be the set of vertices that are initially blue. At the beginning (that is, before the player makes her move) of time step $t \ge 1$, let $B_t$ denote the set of all blue vertices (either active, dormant or extinct), $R_t \subseteq B_t$ denote the set of vertices that are active (blue vertices that have not yet performed a force whose first neighbourhood is entirely blue and that have at least one white vertex in the second neighbourhood), and $W_t$ denote the set of white vertices. Then, at each sequential hop, the player selects a vertex $x_t \in R_t$ and performs a single force by hopping from $x_t$ to $y_t \in N_2(x_t) \cap W_t$, where $N_2(x_t)$ denotes the second neighbourhood of $x_t$ , that is, the set of vertices that are at distance exactly two from $x_t$. Subsequently, the sets $B_t$, $R_t$, and $W_t$ are updated accordingly. In particular, 
 	\begin{align*}
		B_{t+1} = B_t\cup \{y_t\} \\
		W_{t+1} = W_t\setminus \{y_t\}.
	\end{align*}
 Clearly, vertex $x_t$ needs to be removed from $R_t$ since it became extinct but more updates might be necessary as potentially more vertices change their status from active to dormant or vice versa.

    Note that a successful \textbf{strategy} $\scr{S}$ can be defined as an initial set $B_1$ of blue vertices and a sequence of $\ell$ hops, $x_1 \rightarrow y_1$, $x_2 \rightarrow y_2, \ldots, x_{\ell} \rightarrow y_{\ell}$ where, at the end of the process, every vertex is blue. Such initial sets $B_1$ will be called \textbf{feasible}. Clearly, $B_1 = V(G)$ is feasible which shows that the hopping forcing number is well defined and that $H(G) \le |V(G)|$. On the other hand, $H(G) \ge 1$ since $B_1 = \emptyset$ is infeasible (unless $G$ is the null graph, that is, $V(G)=\emptyset$ in which case, trivially, $H(G)=0$). Furthermore, out of several feasible sets $B_1^1, B_1^2,\ldots$, the set(s) with least cardinality can be called \textbf{optimal} and this cardinality is the hopping number $H(G)$. Observe that given a graph $G$ and $B_1$, both feasibility and optimality can be ascertained at the beginning of the process, that is, this is a `one person game' with perfect information and no randomness. Finally, since the optimal sequential hopping starts with $|V(G)| - H(G)$ white vertices and at each step one of them changes colour to blue, the length of the process is equal to $|V(G)| - H(G)$.
    
    More importantly, the set of initial blue vertices, $B_1$, can be updated in an online fashion, that is, we may try to perform the desired sequence of hops and append some vertices to the initial set dynamically instead of knowing the set $B_1$ a priori. Indeed, we may start with $B_1 = \emptyset$. Then, at time $t$, if $x_t$ is not extinct and $y_t$ is white and in $N_2(x_t)$, then we may add $x_t$ (if needed) and all white neighbours of $x_t$ (if there are any) to $B_1$ so that the desired force can be performed.
		
	\subsection{Main Results} 
	
In this paper, we establish various asymptotic upper and lower bounds for the hopping number of random $d$-regular graph $\mathcal{G}_{n,d}$ (see Subsection~\ref{sec:config} for the definition and more details on this model). We say that a random graph has property $P$ \emph{asymptotically almost surely} (or a.a.s.) if the probability that it has property $P$ tends to $1$ as $n$ goes to infinity (see Subsection~\ref{sec:notation} for more on asymptotic notation used in this paper).

\bigskip

Understanding random $2$-regular graphs is easy. In Section~\ref{sec:2reg}, we prove the following result.
	\begin{theorem}\label{thm:2reg}
	A.a.s.\ $H(\mathcal{G}_{n,2}) \sim (3/2) \log n$.
	\end{theorem}

Unfortunately, the $d=2$ case is the only value of $d \ge 2$ for which we determine an asymptotic behaviour of the hopping number. For the remaining ones, we only have some upper and lower bounds---see Table~\ref{tab:UpperAndLowerBoundsHoppingNumber} and Figure~\ref{fig:UpperAndLowerBoundsHoppingNumber}.

{\scriptsize
\begin{table}[h]
\centering
\noindent
\begin{tabular}{N N N N }\toprule
\multicolumn{1}{N }{\textbf{}} & \multicolumn{2}{c }{\textbf{Lower Bounds}} & \multicolumn{1}{c }{\textbf{Upper Bound}}  \\ 
\cmidrule(lr){2-3}
\cmidrule(ll){4-4}
\multicolumn{1}{ N }{\textbf{Degree}} & \multicolumn{1}{ N }{\textbf{Expander Mixing Lemma}} & \multicolumn{1}{ N }{\textbf{Configuration Model}} &  \multicolumn{1}{ N }{\textbf{Contiguous Model}} \\ 
\cmidrule(lr){1-1}
\cmidrule(lr){2-3}
\cmidrule(ll){4-4}
\multicolumn{1}{ c }{$d=3$} & 0.0149 & 0.0699 &  0.3333  \\ 
\multicolumn{1}{ c }{$d=4$} & 0.0372 & 0.1451  &  0.4571 \\ 
\multicolumn{1}{ c }{$d=5$} & 0.0588 & 0.2114 &  0.5341\\ 
\multicolumn{1}{ c }{$d=6$} & 0.0787 & 0.2678 & 0.5884 \\ 
\multicolumn{1}{ c }{$d=7$} & 0.0968 & 0.3158 &  0.6294 \\ 
\multicolumn{1}{ c }{$d=8$} & 0.1134 & 0.3569 &  0.6618  \\ 
\multicolumn{1}{ c }{$d=9$} & 0.1287 & 0.3924 & 0.6882 \\ 
\multicolumn{1}{ c }{$d=10$} & 0.1429 & 0.4235 &   0.7101\\ 
\cmidrule(ll){1-4}
\multicolumn{1}{ c }{$d=20$} & 0.2445 & 0.6054 &   0.8231 \\ 
\multicolumn{1}{ c }{$d=40$} & 0.3755 & 0.7437 &   0.8946\\ 
\multicolumn{1}{ c }{$d=80$} & 0.5556 & 0.8409 &   0.9386 \\ 
\multicolumn{1}{ c }{$d=160$} & 0.6848 & 0.9048 &  0.9649\\ 
\multicolumn{1}{ c }{$d=320$} & 0.7767 & 0.9446 &  0.9803\\ 
\multicolumn{1}{ c }{$d=640$} & 0.8420 & 0.9684 &  0.9890\\ 
\multicolumn{1}{ c }{$d=1280$} & 0.8882 & 0.9823 &   0.9940\\ 
\bottomrule
\end{tabular}
\caption{Comparison of upper and lower bounds for the hopping number for small and large values of $d$.}
\label{tab:UpperAndLowerBoundsHoppingNumber}
\end{table}
}

\begin{figure}
        \centering
       \includegraphics[scale=0.7]{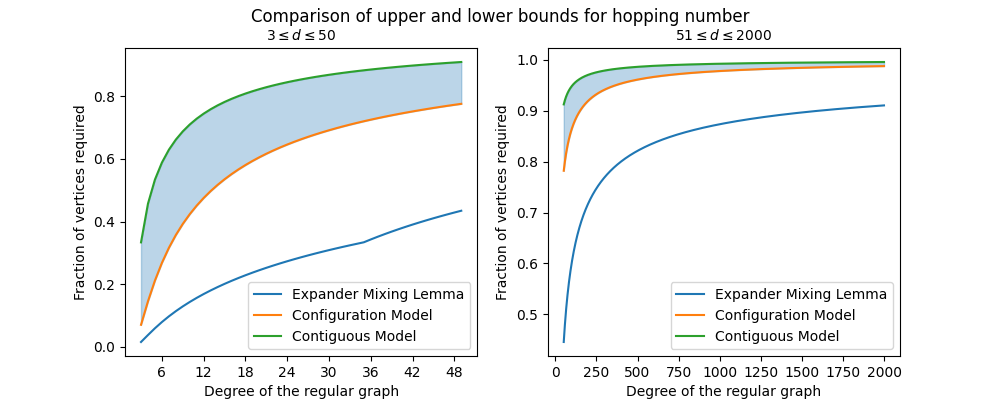}
        \caption{Comparison of upper and lower bounds for the hopping number for small and large values of $d$.}
        \label{fig:UpperAndLowerBoundsHoppingNumber}
\end{figure}
   
Upper bounds are studied in Section~\ref{sec:ge3reg}. An (on-line) algorithm to that we use to create the initial set of blue vertices is universal but the analysis is slightly different in the degenerate case $d=3$. We prove the following upper bound in Subsection~\ref{sec:contiguous_3}.

 	\begin{theorem}\label{thm:contiguous_3}
 A.a.s.\ $H(\mathcal{G}_{n,3}) \le (1+o(1)) \, n/3$.
	\end{theorem}

In Subsection~\ref{sec:contiguous_4}, we analyze the algorithm for $d \ge 4$ yielding the following upper bounds:

 	\begin{theorem}\label{thm:contiguous_4andabove}
	For any integer $d \ge 4$, a.a.s.\ $$H(\mathcal{G}_{n,d}) \le (1+o(1)) \,\frac{(d-1)!(d-2)^{d-1}}{\prod_{i=1}^{d-1}(i(d-2)+1)} \, n.$$
	\end{theorem}

	Note that the constant in the upper bound can be estimated as follows:
	\begin{eqnarray*}
	\frac{(d-1)!(d-2)^{d-1}}{\prod_{i=1}^{d-1}(i(d-2)+1)} &=& \prod_{i=1}^{d-1} \frac{i(d-2)}{i(d-2)+1} ~=~ \prod_{i=1}^{d-1} \left(1 - \frac{1}{i(d-2)+1} \right) \\
	&=& \exp \left( - \Theta \left( \sum_{i=1}^{d-1} \frac{1}{id} \right) \right) ~=~ \exp \left( - \Theta \left( \frac{\log d}{d} \right) \right) \\
	&=& 1 - \Theta \left( \frac{\log d}{d} \right).
	\end{eqnarray*}
	In particular, it shows that it tends to one as $d \to \infty$. 
	
	\medskip
	
	To get a lower bound for $H(\mathcal{G}_{n,d})$ that explicitly tends to one as $d \to \infty$, we use the expansion properties of random $d$-regular graphs to get a lower bound of $1 - \Theta ( 1 / \sqrt{d} )$. Indeed, in Section~\ref{sec:lower_bounds}, we prove the following result. (See Subsection~\ref{sec:def_expansion} for the definition of $\lambda(G)$.)

\begin{theorem}\label{upperBoundExpanderMixing}
Let $G=(V,E)$ be a $d$-regular graph with $n$ vertices and set $\lambda = \lambda(G)$. Then,
$$
H(G) \ge \max \left( 1 - \frac {2\lambda}{d}, \frac {d-\lambda}{d+3\lambda} \right) n ~~=~~ \left( 1 - \min \left( \frac {2 \lambda}{d}, \frac {4\lambda}{d+3\lambda} \right) \right) n.
$$ 

As a result, for any $d \ge 3$ and $\eps > 0$, a.a.s. 
$$
H(\mathcal{G}_{n,d}) \ge \left( 1 - \min \left( \frac {4 \sqrt{d-1}}{d}, \frac {8\sqrt{d-1}}{d+6\sqrt{d-1}} \right) -\eps \right) n.
$$ 
\end{theorem}

In Section~\ref{sec:lower_bounds_config}, the above lower bound is strengthened by applying the configuration model to get the following, stronger but implicit and numerical, lower bound.
\begin{theorem}\label{upperBoundConfig}
For a given integer $d \ge 3$, let 
\begin{eqnarray*}
g_d(x,z) &=& \left( \frac {d}{2} -1 - dz \right) x \log(x) + (d-1) (1-x) \log \left( \frac {1-x}{2} \right) \nonumber \\
&& - \, 2dxz\log(z) \nonumber \\
&& - \, \frac {(1-2z)dx}{2} \log(1-2z) \nonumber \\
&& - \, d \left( \frac{1-x}{2}-zx \right) \log \left( \frac{1-x}{2}-zx \right). \label{eq:gd_function}
\end{eqnarray*}
For a fixed $x \in (0,1)$, function $g_d(x,z)$ is maximized at
\begin{equation*}
z_0(x) := \frac {1 - \sqrt{ 1-2(1-x)x }}{2x}\,. \label{eq:z0}
\end{equation*}
 Fix $\eps > 0$. Let $x_d$ be the unique $x \in (0,1)$ for which $h_d(x)=g_d(x,z_0(x))=0$. Then, a.a.s.\ 
$$
H(\mathcal{G}_{n,d}) > (x_d - \eps)n. 
$$
\end{theorem}

\bigskip

We finish the paper with our attempt to get better upper bounds for the hopping number by introducing a degree-greedy algorithm~\cite{wormald2003analysis} to create the initial set of blue vertices and then use the differential equation method to analyze it. Unfortunately, the bounds we obtained using this method turned out to be weaker than the ones we established above. Nevertheless, in Section~\ref{sec:DEattempt} we briefly report our attempt for the case of random 3-regular graphs with the hope that one can modify our algorithm, and use similar techniques to analyze it, to get better bounds than the ones we managed to prove.

	\section{Preliminaries}
	
	\subsection{Notation}\label{sec:notation}
	
	The results presented in this paper are asymptotic by nature. We say that a random graph has property $P$ \emph{asymptotically almost surely} (or a.a.s.) if the probability that it has property $P$ tends to $1$ as $n$ goes to infinity.  Given two functions $f=f(n)$ and $g=g(n)$, we will write $f(n)=\bigo(g(n))$ if there exists an absolute constant $c \in \R_+$ such that $|f(n)| \leq c|g(n)|$ for all $n$, $f(n)=\Omega(g(n))$ if $g(n)=\bigo(f(n))$, $f(n)=\Theta(g(n))$ if $f(n)=\bigo(g(n))$ and $f(n)=\Omega(g(n))$, and we write $f(n)=o(g(n))$ or $f(n) \ll g(n)$ if $\lim_{n\to\infty} f(n)/g(n)=0$. In addition, we write $f(n) \gg g(n)$ if $g(n)=o(f(n))$ and we write $f(n) \sim g(n)$ if $f(n)=(1+o(1))g(n)$, that is, $\lim_{n\to\infty} f(n)/g(n)=1$.
	
	We will use $\log n$ to denote a natural logarithm of $n$. For a given $n \in \Nn := \{1, 2, \ldots \}$, we will use $[n]$ to denote the set consisting of the first $n$ natural numbers, that is, $[n] := \{1, 2, \ldots, n\}$. Finally, as typical in the field of random graphs, for expressions that clearly have to be an integer, we round up or down but do not specify which: the choice of which does not affect the argument.
	
	\subsection{Random $d$-regular Graphs}\label{sec:config}
	
	Our main results refer to the probability space of \emph{random $d$-regular graphs} with uniform probability distribution. This space is denoted $\mathcal{G}_{n,d}$, and asymptotics are for $n\to\infty$ with $d\ge 2$ fixed, and $n$ even if $d$ is odd.
	
	Instead of working directly in the uniform probability space of random regular graphs on $n$ vertices $\mathcal{G}_{n,d}$, we use the \textit{configuration model} of random regular graphs, first introduced by Bollob\'{a}s~\cite{bollobas1980probabilistic}, which is described next. Suppose that $dn$ is even, as in the case of random regular graphs, and consider $dn$ points partitioned into $n$ labeled buckets $v_1,v_2,\ldots,v_n$ of $d$ points each. A \textit{pairing} of these points is a perfect matching into $dn/2$ pairs. Given a pairing $P$, we may construct a multigraph $G(P)$, with loops and parallel edges allowed, as follows: the vertices are the buckets $v_1,v_2,\ldots,v_n$, and a pair $\{x,y\}$ in $P$ corresponds to an edge $v_iv_j$ in $G(P)$ if $x$ and $y$ are contained in the buckets $v_i$ and $v_j$, respectively.
	
	It is an easy fact that the probability of a random pairing corresponding to a given simple graph $G$ is independent of the graph, hence the restriction of the probability space of random pairings to simple graphs is precisely $\mathcal{G}_{n,d}$. Moreover, it is well known that a random pairing generates a simple graph with probability asymptotic to $e^{(1-d^2)/4}$ depending on $d$, so that any event holding a.a.s.\ over the probability space of random pairings also holds a.a.s.\ over the corresponding space $\mathcal{G}_{n,d}$. For this reason, asymptotic results over random pairings suffice for our purposes. One of the advantages of using this model is that the pairs may be chosen sequentially so that the next pair is chosen uniformly at random over the remaining (unchosen) points. For more information on this model, see the survey~\cite{wormald1999models} or any of the books on random graphs~\cite{Bol01,JLR00,KF16}.
	
	\subsection{Contiguous Model}\label{sec:contiguous}
	
	The notion of the union of two random regular graphs on the same vertex set is very useful for proving asymptotic properties of $\mathcal{G}_{n,d}$ with $d \ge 3$. In particular, it is known that, for the purpose of proving statements a.a.s., such a random graph can be viewed as the multigraph formed from the union of a Hamilton cycle and random $(d-2)$-regular graph on the same vertex set; see~\cite[Theorem~4.15]{wormald1999models} for a stronger and more general result. (The probability of multiple edges being created is bounded away from 1, and the resulting graph, conditional upon no multiple edges, is contiguous to a random $d$-regular graph.) 
	
	\subsection{Expansion Properties of Random $d$-regular Graphs}\label{sec:def_expansion}
	
	We will use the expansion properties of random $d$-regular graphs that follow from their eigenvalues. The adjacency matrix $A=A(G)$ of a given a $d$-regular graph $G$ with $n$ vertices, is an $n \times n$ real and symmetric matrix. Thus, the matrix $A$ has $n$ real eigenvalues which we denote by $\lambda_1 \ge \lambda_2 \ge \cdots \ge \lambda_n$. It is known that certain properties of a $d$-regular graph are reflected in its spectrum but, since we focus on expansion properties, we are particularly interested in the following quantity: $\lambda = \lambda(G) = \max( |\lambda_2|, |\lambda_n|)$. In words, $\lambda$ is the largest absolute value of an eigenvalue other than $\lambda_1 = d$. For more details, see the general survey~\cite{hoory2006expander} about expanders, or~\cite[Chapter 9]{alon2016probabilistic}.
	
	The value of $\lambda$ for random $d$-regular graphs has been studied extensively. A major result due to Friedman~\cite{friedman2008proof} is the following:
	\begin{lemma}[\cite{friedman2008proof}]\label{lem:Fri}
		For every fixed  $\varepsilon > 0$ and for $G\in \mathcal{G}_{n,d}$,
		$$
		\Prob( \lambda(G) \le 2 \sqrt{d-1}+ \varepsilon) = 1 - o(1)\,.
		$$
	\end{lemma}
	
	The number of edges $|E(S,T)|$ between sets $S$ and $T$ is expected to be close to the expected number of edges between $S$ and $T$ in a random graph of edge density $d/n$, namely, $d|S||T|/n$. A small $\lambda$ (or large spectral gap) implies that this deviation is small. The following useful bound is essentially proved in~\cite{alon1988explicit} (see also~\cite{alon2016probabilistic}):
	\begin{lemma}[Expander Mixing Lemma]\label{lem:AC}
		Let $G=(V,E)$ be a $d$-regular graph with $n$ vertices and set $\lambda = \lambda(G)$. Then for all $S, T \subseteq V$,
		$$
		\left| |E(S, T)| - \frac{d|S||T|}{n} \right| \le \lambda \sqrt{|S||T|} \,.
		$$
	\end{lemma}
	\noindent
	(Note that $S \cap T$ does not have to be empty; in general, $|E(S,T)|$ is defined to be the number of edges between $S \setminus T$ to $T$ plus twice the number of edges that contain only vertices of $S \cap T$.)
	
    At some point it will be better to apply stronger estimates for $|E(S,V \setminus S)|$ that can be easily derived from a slightly stronger version of the above lemma for $|E(S,S)|$ (see~\cite{alon1988explicit}), namely,
	$$
		\left| |E(S, S)| - \frac{d|S|^2}{n} \right| \le \frac {\lambda |S||V \setminus S|}{n}   
    $$
	for all $S \subseteq V$. Since $|E(S,V \setminus S)| = d|S| - |E(S,S)|$, we immediately get that
    \begin{equation} \label{eqn:bisection}
		\left| |E(S,V \setminus S)| - \frac{d|S||V \setminus S|}{n} \right| \le \frac{\lambda|S||V \setminus S|}{n}
	\end{equation}
	for all $S \subseteq V$. 
	
	\subsection{Concentration Tools}\label{sec:concentration}
	
	Let us first state a few specific instances of Chernoff's bound that we will find useful. Let $X \in \textrm{Bin}(n,p)$ be a random variable distributed according to a Binomial distribution with parameters $n$ and $p$. Then, a consequence of \emph{Chernoff's bound} (see e.g.~\cite[Theorem~2.1]{JLR}) is that for any $t \ge 0$ we have
	\begin{eqnarray}
		\Prob( X \ge \E X + t ) &\le& \exp \left( - \frac {t^2}{2 (\E X + t/3)} \right)  \label{chern1} \\
		\Prob( X \le \E X - t ) &\le& \exp \left( - \frac {t^2}{2 \E X} \right).\label{chern}
	\end{eqnarray}
	
	Moreover, let us mention that the bound holds in a more general setting as well, that is, for $X=\sum_{i=1}^n X_i$ where $(X_i)_{1\le i\le n}$ are independent variables and for every $i \in [n]$ we have $X_i \in \textrm{Bernoulli}(p_i)$ with (possibly) different $p_i$-s (again, see~e.g.~\cite{JLR} for more details). Finally, it is well-known that the Chernoff bound also applies to negatively correlated Bernoulli random variables~\cite{dubhashi1998balls}.
	
	\subsection{The Differential Equation Method}\label{sec:DEmethod}
	
	In this paper, we will use the differential equation method (see~\cite{BD20} for a gentle introduction) to establish dynamic concentration of our random variables. The origin of the differential equation method stems from work done at least as early as 1970 (see Kurtz~\cite{Kurtz1970}), and which was developed into a very general tool by Wormald~\cite{W1995,W1999} in the 1990's. Indeed, Wormald proved a ``black box'' theorem, which gives dynamic concentration so long as some relatively simple conditions hold. Warnke~\cite{Warnke2020} recently gave a short proof of a somewhat stronger black box theorem.
	
	In this section, we provide a self-contained \textit{non-asymptotic} statement of the differential equation method which we will use for each property we investigate. The statement combines~\cite[Theorem $2$]{Warnke2020}, and its extension~\cite[Lemma $9$]{Warnke2020}, in a form convenient for our purposes, where we modify the notation of~\cite{Warnke2020} slightly. In particular, we rewrite~\cite[Lemma $9$]{Warnke2020} in a less general form in terms of a stopping time $T$. We need only check the `Boundedness Hypothesis' (see below) for $0 \le t \le T$, which is exactly the setting in our proofs.
	
	Suppose we are given integers $a,n \ge 1$, a bounded domain $\scr{D} \subseteq \mb{R}^{a+1}$, and functions $(F_k)_{1 \le k \le a}$ where each $F_k: \scr{D} \to \mb{R}$ is $L$-Lipschitz-continuous on $\scr{D}$ for $L \ge 0$. Moreover, suppose that $R \in [1, \infty)$ and $S \in (0, \infty)$ are \textit{any} constants which satisfy $\max_{1 \le k \le a} |F_{k}(x)| \le R$ for all $x=(s,y_1,\ldots ,y_{a})\in \scr{D}$ and $0 \le s \le S$.
	
	\begin{theorem}[Differential Equation Method, \cite{Warnke2020}] \label{thm:differential_equation_method}
		Suppose we are given $\sigma$-fields $\scr{F}_{0}  \subseteq \scr{F}_{1} \subseteq \cdots$, and for each $t \ge 0$, random variables $((Y_{k}(t))_{1 \le k \le a}$ which are $\scr{F}_t$-measurable. Define $T_{\scr{D}}$ to be the minimum $t \ge 0$ such that
		\[
		(t/n, Y_{1}(t)/n, \ldots , Y_{a}(t)/n) \notin \scr{D}.
		\]
		Let $T \ge 0$ be an (arbitrary) stopping time\footnote{The stopping time $T\ge 0$ is \textbf{adapted} to $(\scr{F}_t)_{t \ge 0}$, provided the event $\{\tau = t\}$ is $\scr{F}_t$-measurable for each $t \ge 0$.} adapted to $(\scr{F}_t)_{t \ge 0}$, and assume that the following conditions hold for $\delta, \beta, \gamma \ge 0$ and $\lambda \ge \delta \min\{S, L^{-1}\} + R/n$:
		\begin{enumerate}
			\item[(i)] The `Initial Condition': For some $(0,\hat{y}_1,\ldots ,\hat{y}_a) \in \scr{D}$, \label{enum:initial_conditions}
			\[
			\max_{1 \le k \le a} |Y_{k}(0) - \hat{y}_k n| \le \lambda n.
			\] 
			\item[(ii)] The `Trend Hypothesis': For each  $t \le \min\{ T, T_{\scr{D}} -1\}$, \label{enum:trend_hypothesis}
			$$|\mb{E}[ Y_{k}(t+1) - Y_{k}(t) \mid \scr{F}_t] - F_{k}(t/n,Y_{1}(t)/n,\ldots ,Y_{a}(t)/n)| \le \delta.$$
			\item[(iii)] The `Boundedness Hypothesis': With probability $1 - \gamma$, \label{enum:boundedness_hypothesis}
			$$|Y_{k}(t+1) -  Y_{k}(t)| \le \beta,$$
			for each $t \le \min\{ T, T_{\scr{D}} -1\}$.
		\end{enumerate}
		Then, with probability at least $1 - 2a \exp\left(\frac{-n \lambda^2}{8 S \beta^2}\right) - \gamma$, we have that
		\begin{equation}
			\max_{0 \le t \le \min\{T, \sigma n\}} \max_{1 \le k \le a} |Y_{k}(t) -y_{k}(t/n) n| < 3 \lambda \exp(L S)n,
		\end{equation}
		where $(y_{k}(s))_{1 \le k \le a}$ is the unique solution to the system of differential equations
		\begin{equation} \label{eqn:general_de_system}
			y_{k}'(s) = F_{k}(s, y_{1}(s),\ldots ,y_{a}(s)) \quad \mbox{with $y_{k}(0) = \hat{y}_k$ for $1 \le k \le a$,}
		\end{equation}
		and $\sigma = \sigma(\hat{y}_1,\ldots ,\hat{y}_a) \in [0,S]$ is any choice of $\sigma \ge 0$ with the property that $(s,y_{1}(s),\ldots, y_{a}(s))$ has $\ell^{\infty}$-distance at least $3 \lambda \exp(LS)$ from the boundary of $\scr{D}$ for all $s \in [0, \sigma)$.
		\begin{remark}
			Standard results for differential equations guarantee that \eqref{eqn:general_de_system} has a unique solution $(y_{k}(s))_{1\le k \le a}$ which extends arbitrarily close to the boundary of $\scr{D}$. 
		\end{remark}
	\end{theorem}
	
	\section{2-regular Graphs}\label{sec:2reg}
	
	Let $Y = Y(n)$ be the total number of cycles in a random 2-regular graph on $n$ vertices. Since exactly three vertices need to be initially blue in each cycle (that is, $H(C_i) = 3$ for any $i \ge 3$), $H(\mathcal{G}_{n,2}) = 3 Y(n)$.
	
	We know that the random 2-regular graph is a.a.s.\ disconnected; by simple calculations one can show that the probability of having a Hamiltonian cycle is asymptotic to $\frac 12 e^{3/4} \sqrt{\pi/n} = o(1)$ (see, for example,~\cite{wormald1999models}).  We also know that the total number of cycles $Y(n)$ is sharply concentrated around $(1/2) \log n$. Indeed, it is not difficult to see this by generating the random graph sequentially using the pairing model. The probability of forming a cycle in step $i$ is exactly $1/(2n-2i+1)$, so the expected number of cycles is 
	$$
	\sum_{i=1}^n \frac {1}{2n-2i+1}= \sum_{i=1}^{2n} \frac {1}{i} - \frac {1}{2} \sum_{i=1}^{n} \frac {1}{i} = \log(2n) - \frac {1}{2} \log n + O(1) =  \frac {1}{2} \log n + O(1). 
	$$
	The variance can be calculated in a similar way. So we get the following result (Theorem~\ref{thm:2reg}): 
	\begin{center}
	A.a.s.\ $H(\mathcal{G}_{n,2}) \sim (3/2) \log n$.
	\end{center}

	\section{Upper Bounds from the Contiguous Model: $d$-regular Graphs, $d \ge 3$}\label{sec:ge3reg}
	
	To provide an upper bound for $H(\mathcal{G}_{n,d})$ for 3-regular graphs, and subsequently for $d$-regular graphs with $d \ge 4$, we will use the contiguous model introduced in Subsection~\ref{sec:contiguous}. For a given $d \ge 3$, a $d$-regular graph (generated by the pairing model) can be viewed as the union of a Hamilton cycle $(v_1, v_2, \ldots, v_n)$ and random $(d-2)$-regular graph on the same vertex set, namely, $\{v_1, v_2, \ldots, v_n\}$. We will call the two neighbours of $v_k$ that are on the Hamilton cycle HC-neighbours; the remaining $d-2$ neighbours of $v_k$ will be called RG-neighbours. 
	
	As explained at the end of Subsection~\ref{sec:definitions}, it will be easier to update $B_1$, the set of initial blue vertices in an online fashion. The strategy will be the same for all $d \ge 3$ but formulas for $d=3$, the degenerate case, will be slightly different. Therefore, once we explain the general strategy we will independently deal with the $d=3$ case (Subsection~\ref{sec:contiguous_3}) before moving to the $d \ge 4$ case (Subsection~\ref{sec:contiguous_4}). 

\medskip
 
	Start with $B_1 = \emptyset$. We attempt to colour vertices blue as we hop along the Hamilton cycle. We start by turning $v_1$ and all of its neighbours blue; these vertices are added to $B_1$. We will try to make $v_1$ to hop to some neighbour of $v_2$. (Of course, $v_1$ can try to hop through some other neighbour, not necessarily through $v_2$, but insisting on this choice will make the analysis of the strategy tractable. But this strategy is certainly suboptimal.) If $v_3$ is white, then $v_1$ can hop there and force $v_3$ to become blue (note that $v_3$ could be a neighbour of $v_1$ and so could be blue). Similarly, if any of the $(d-2)$ RG-neighbours of $v_2$ are white, then $v_1$ can hop and force one of them to become blue. Note that if $v_1$ could not hop through $v_2$, then all neighbours of $v_2$ are already blue. Otherwise, we turn the remaining white neighbours of $v_2$ (if there are any) blue; these vertices are added to $B_1$. After that we will try to make $v_2$ to hop through $v_3$ to some neighbour of $v_3$, and continue hoping along the Hamilton cycle. Once we investigate $v_{n-3}$, the strategy is finished and we can check how many vertices were added to $B_1$ during this process. (Note that $v_n$ is a HC-neighbour of $v_1$ and so when we reach $v_{n-2}$, all vertices are certainly blue.) By design, the set $B_1$ that is constructed during this process is feasible and so its size yields the desired upper bound for $H(\mathcal{G}_{n,d})$. 
	
	Before we describe the situation when $v_t$ tries to hop through $v_{t+1}$, let us make a simple observation that will simplify the analysis of the above process. Let $X$ be the random variable counting how many vertices could not hop during the process; clearly, $Y = n-X$ is the number of vertices that hopped. Since each time a vertex hops exactly one white vertex turns blue, $|B_1|+Y=n$ which is equivalent to 
	$$
	|B_1| = n - Y = X.
	$$
Hence, if we prove that a.a.s.\ $X \sim f(n)$ for some deterministic function $f(n)$, then we may conclude that a.a.s.\ $H(\mathcal{G}_{n,d}) \le (1+o(1)) f(n)$. 

\medskip

	For each $t \in [n-3]$, let $X_t$ be the indicator random variable for the event that $v_t$ cannot hop through $v_{t+1}$. Clearly, $X = \sum_{t=1}^{n-3} X_t$. When $v_t$ tries to hop though $v_{t+1}$, vertices $v_1, v_2, \ldots, v_t$ and all of their neighbours are blue (in particular, $v_{t+1}$ is blue); the remaining vertices are white. Vertex $v_t$ cannot hop (that is, $X_t = 1$) if and only if the following two properties hold:
	\begin{itemize}
		\item [(P1)] at least one RG-neighbour of $v_{t+2}$ is in $\{ v_1, v_2, \ldots, v_t \}$ (that is, $v_{t+2}$ is blue so $v_t$ cannot hop there),
		\item [(P2)] for some $i \in \{ 0, 1, \ldots, d-2 \}$,
			\begin{itemize}
				\item [(P2')] $i$ RG-neighbours of $v_{t+1}$ are in $\{ v_n, v_1, v_2, \ldots, v_t, v_{t+2} \}$, and
				\item [(P2'')] $(d-2-i)$ RG-neighbours of $v_{t+1}$ are in $\{ v_{t+3}, v_{t+4}, \ldots, v_{n-1} \}$ but all of these $(d-2-i)$ RG-neighbours have at least one RG-neighbour in $\{ v_1, v_2, \ldots, v_t \}$
			\end{itemize}
			(that is, all of the $(d-2)$ RG-neighbours of $v_{t+1}$ are blue so $v_t$ cannot hop to any of these vertices).
	\end{itemize}

	\subsection{$d=3$ case}\label{sec:contiguous_3}

	The case $d=3$ is the degenerate case (slightly different and easier to analyze), and we will deal with it independently. We will prove the following (Theorem~\ref{thm:contiguous_3}):
\begin{center}
 A.a.s.\ $H(\mathcal{G}_{n,3}) \le (1+o(1)) \, n/3$.
 \end{center}
	\begin{proof}[Proof of Theorem~\ref{thm:contiguous_3}]
		Fix any $t \in [n-3]$. To compute the probability that $X_t=1$, we first expose the unique RG-neighbour of $v_{t+2}$. Property~(P1) holds with probability $t / (n-1)$. Conditioning on this event, we expose the RG-neighbour of $v_{t+1}$ to determine whether property~(P2) holds or not. An important observation is that if the unique RG-neighbour of $v_{t+1}$ is in $\{ v_{t+3}, v_{t+4}, \ldots, v_{n-1} \}$, then the property~(P2'') cannot hold---this RG-neighbour has only one RG-neighbour, namely, $v_{t+1}$ which is not in $\{ v_1, v_2, \ldots, v_t \}$. In other words, the only chance that property~(P2) holds is when $i = 1$. This makes the case $d=3$ degenerate and distinguishes it from the case $d \ge 4$. The conditional probability that~(P2) holds is then equal to $t/(n-3)$. We get that 
		$$
		\Prob(X_t = 1) = \frac {t}{n-1} \cdot \frac {t}{n-3}
		$$
		and so
		\begin{eqnarray}\label{expectation}
		\E [X] & = & \sum_{t=1}^{n-3} \Prob(X_t = 1)
		~=~ \sum_{t=1}^{n-3} \frac{t^2}{(n-1)(n-3)} \nonumber\\
		&=& \frac{1}{(n-1)(n-3)} \cdot \frac{(n-3)(n-4)(2n-5)}{6} 
		~=~ \frac{n}{3} + \bigo(1).
		\end{eqnarray}
	
		It remains to show that $X$ is well-concentrated around its expectation. We demonstrate this by estimating the variance. First, note that
	\begin{eqnarray*}
		\Var[X] & = & \sum_{\substack{1\leq k < \ell \leq n-3}} \big( \Prob(X_k = X_\ell = 1) - \Prob(X_k=1)\Prob(X_\ell = 1) \big) \\
		&& + \sum_{k = 1}^{n-3} \big( \Prob(X_k=1) - \Prob(X_k = 1)^2 \big). 
  	\end{eqnarray*}
By~(\ref{expectation}), the second term is at most $\E [X] = \bigo(n)$. Moreover, the first term can be split further depending on whether $\ell = k+1$ or $\ell \geq k+2$. If $\ell = k+1$, then each term is trivially at most one, thus the corresponding sum is again $\bigo(n)$. It follows that 
   \begin{eqnarray*}
\Var[X] &=& \bigo(n) + \sum_{\substack{1\leq k < \ell \leq n-3 \\ \ell \geq k+2}} \big( \Prob(X_k = X_\ell = 1) - \Prob(X_k=1)\Prob(X_\ell = 1) \big) \\
&=& \bigo(n) + \sum_{\substack{1\leq k < \ell \leq n-3 \\ \ell \geq k+2}} \left( \frac{k^2}{(n-1)(n-3)} \cdot \frac { (\ell+\bigo(1))^2}{(n-5)(n-7)} - \frac{k^2}{(n-1)(n-3)} \cdot \frac{\ell^2}{(n-1)(n-3)} \right) \\
&=& \bigo(n) + \sum_{k=1}^{n-5} \bigo\Big(\frac{k^2}{n^2}\Big) \sum_{\ell=k+2}^{n-3} \bigo\Big(\frac{\ell}{n^2}\Big) \\
&=& \bigo(n) + \sum_{k=1}^{n-5} \bigo\Big(\frac{k^2}{n^2}\Big) 
~=~ \bigo(n) ~=~ o(n^2).
	\end{eqnarray*}
 Since $\Var[X]  =o((\E [X])^2) $, $X$ is well-concentrated around its expectation by the second moment method. The proof of the theorem is finished.
	\end{proof}

	\subsection{$d \geq4$ case}\label{sec:contiguous_4}

 The bound from Theorem~\ref{thm:contiguous_4andabove} is replicated here for easier reference:
 for any integer $d \ge 4$, a.a.s.
 \begin{eqnarray}
     	\ H(\mathcal{G}_{n,d}) \le (1+o(1)) \,\frac{(d-1)!(d-2)^{d-1}}{\prod_{i=1}^{d-1}(i(d-2)+1)} \, n.
 \end{eqnarray}
  The exact values for $3\leq d \leq 10$ are presented in Table~\ref{tab:upper_bound}. Approximated values can be found in  Table~\ref{tab:UpperAndLowerBoundsHoppingNumber}.
		\begin{table}[htp]
		\caption{Explicit upper bounds for the hopping number for small values of $d$.}
		\begin{center}
			\begin{tabular}{|c|c|c|c|}
				\hline
				Degree & Upper Bound & Degree & Upper Bound \\
				\hline
				$d=3$ & 1/3 & $d=7$ & 78125/124124 \\
				$d=4$ & 16/35 & $d=8$ & 40310784/60911435 \\
				$d=5$ & 243/455 & $d=9$ & 40353607/58640175 \\
				$d=6$ & 8192/13923 & $d=10$ & 17179869184/24192643475 \\
				\hline
			\end{tabular}
		\end{center}
		\label{tab:upper_bound}
	\end{table}
			 
    		The case of  $d \geq 4$ follows a similar structure to the case where $d=3$. The main distinction here is that, since there are always at least two RG-neighbours in addition to the two HC-neighbours on the cycle for $d \geq 4$, a vertex could be a RG-neighbour of more than one vertex in $\{v_1, v_2, \ldots, v_t\}$ (see Property~(P1)). Similarly, $v_{t+1}$ could have $(d-2-i)$ RG-neighbours in $\{ v_{t+3}, v_{t+4}, \ldots, v_{n-1} \}$ \emph{and} all of these RG-neighbours can have at least one RG-neighbour in $\{ v_1, v_2, \ldots, v_t \}$ for all values of $i\in\{ 0, 1, \ldots, d-2 \}$. In other words, the Property~(P2) holds with a non-zero probability even when $i=0$. 
	
	\begin{proof}[Proof of Theorem~\ref{thm:contiguous_4andabove}]
  Fix any $t \in [n-3]$. To compute the probability that $X_t=1$, we first expose the $(d-2)$ RG-neighbours of $v_{t+2}$, one by one. Property~(P1) holds with probability 
  \begin{equation}\label{eq:p1}
  1- \prod_{i=1}^{d-2} \left(1- \frac {(d-2)t}{(d-2)n-2i+1} \right) + \bigo(n^{-1}) = 1 - \left(1- \frac {t}{n} \right)^{d-2} + \bigo(n^{-1}). 
  \end{equation}
  (The first $\bigo(n^{-1})$ term corresponds to the event that there is a loop at $v_{t+2}$.) Conditioning on Property~(P1), we expose the RG-neighbours of $v_{t+1}$ to determine whether property~(P2) holds or not. The conditional probability that~(P2) holds is equal to
  \begin{align}
   & \bigo(n^{-1}) + \sum_{i=0}^{d-2} {{d-2}\choose {i}} \left( \frac {(d-2) t+\bigo(1)}{(d-2)n+\bigo(1)} \right)^i \left( 1- \frac {(d-2)t+\bigo(1)}{(d-2)n+\bigo(1)} \right)^{d-2-i} \nonumber \\
   & \qquad \qquad \qquad \qquad \cdot \left( 1- \left( 1- \frac {(d-2)t+\bigo(1)}{(d-2)n+\bigo(1)} \right)^{d-3} \right)^{d-2-i} \nonumber \\ 
   & = \bigo(n^{-1}) +  \sum_{i=0}^{d-2} {{d-2}\choose {i}} \left( \frac {t}{n} \right)^i \left( 1- \frac {t}{n} \right)^{d-2-i} \left( 1- \left( 1- \frac {t}{n} \right)^{d-3} \right)^{d-2-i}. \label{eq:p2}
  \end{align}
  We get that 
		\begin{align*}
		\Prob(X_t = 1) & = \bigo(n^{-1}) + \Bigg(1-\Big(1-\frac {t}{n}\Big)^{d-2}\Bigg) \\
  & \qquad \qquad \qquad \cdot \sum_{i=0}^{d-2} {{d-2}\choose {i}} \Big( \frac{t}{n} \Big)^i \Big(1-\frac{t}{n} \Big)^{d-2-i}\Bigg(1-\Big(1-\frac{t}{n}\Big)^{d-3}\Bigg)^{d-2-i} \\
  & = \bigo(n^{-1}) + \Bigg(1-\Big(1-\frac {t}{n}\Big)^{d-2}\Bigg) \cdot \Bigg(\frac{t}{n} + \Big(1-\frac{t}{n}\Big)-\Big(1-\frac{t}{n}\Big)^{d-2}\Bigg)^{d-2} \\
    & = \bigo(n^{-1}) + \Bigg(1-\Big(1-\frac {t}{n}\Big)^{d-2}\Bigg)^{d-1}
		\end{align*}
		and so
		\begin{eqnarray}\label{expectationd4}
		\E [X] & = & \sum_{t=1}^{n-3} \Prob(X_t = 1) 
		~=~ \bigo(1)+ \sum_{t=1}^{n-3} \Bigg(1-\Big(1-\frac {t}{n}\Big)^{d-2}\Bigg)^{d-1} \nonumber\\
		&=& \bigo(1)+ n\int_0^1 (1-(1-x)^{d-2})^{d-1} \,dx \nonumber\\
		&=&  \bigo(1) + n\frac{(d-1)!(d-2)^{d-1}}{\prod_{i=1}^{d-1}(i(d-2)+1)},
		\end{eqnarray}
	where the last integral follows via a recursive reduction approach. 
 	As with the case where $d=3$, we would like to show that $X$ is well-concentrated around its expectation by estimating the variance. First, note that
	\begin{eqnarray*}
		\Var[X] & = & \sum_{\substack{1\leq k < \ell \leq n-3}} \big( \Prob(X_k = X_\ell = 1) - \Prob(X_k=1)\Prob(X_\ell = 1) \big) \\
		&& + \sum_{k = 1}^{n-3} \big( \Prob(X_k=1) - \Prob(X_k = 1)^2 \big). 
  	\end{eqnarray*}
By~(\ref{expectationd4}), the second term is at most $\E[X] = \bigo(n)$. As in the proof of Theorem~\ref{thm:contiguous_3}, the first term can be split further depending on whether $\ell = k+1$ or $\ell \geq k+2$. If $\ell = k+1$, then each term is trivially at most one, thus the sum is again $\bigo(n)$. To estimate $\Prob(X_k = X_\ell = 1)$ we first expose the $(d-2)$ RG-neighbours of $v_{k+2}$ and compute the probability of Property~(P1) holding for $X_k$ as in~(\ref{eq:p1}). Conditioning on that, Property~(P1) holds for $X_\ell$ with probability
$$
  1- \prod_{i=1}^{d-2} \left(1- \frac {(d-2)\ell + \bigo(1)}{(d-2)n-2i-2(d-2)+1} \right) + \bigo(n^{-1}) = 1 - \left(1- \frac {\ell}{n} \right)^{d-2} + \bigo(n^{-1}). 
$$
Conditioning on that, the probability that Property~(P2) holds for both $X_k$ and $X_\ell$ can be computed the same way as in~(\ref{eq:p2}). Clearly, the fact that $X_k=1$ affects the probability that $X_\ell=1$ but the difference is hidden in the $\bigo(1)$ terms. We get that
$$
 \Prob(X_k = X_\ell = 1) = \bigo(n^{-1}) + \Bigg(1-\Big(1-\frac {k}{n}\Big)^{d-2}\Bigg)^{d-1} \cdot \Bigg(1-\Big(1-\frac {\ell}{n}\Big)^{d-2}\Bigg)^{d-1}. 
$$
It follows that $\Prob(X_k = X_\ell = 1) - \Prob(X_k=1)\Prob(X_\ell = 1) = \bigo(n^{-1})$ and so $\Var[X] = \bigo(n)$. 
  Since $\Var[X]  =o((\E [X])^2) $, $X$ is well-concentrated around $\E[X]$ by the second moment method. The proof of the theorem is finished. 
\end{proof}
  
	\section{Lower Bounds From the Expansion Properties: $d$-regular Graphs, $d \ge 3$}\label{sec:lower_bounds}
	
In this section, we provide rudimentary lower bounds by invoking the Expander Mixing Lemma. However, let us first present an observation that will be used not only here but also for stronger numerical bounds presented in the next section. 

\begin{lemma}\label{lem:big_sets}
Suppose that $H(G) \le k$ for some graph $G=(V,E)$. Then, $V$ can be partitioned into sets $S, T$, and $U$ such that $|S|=|T|=(n-k)/2$, $|U|=k$, and there is no edge between $S$ and~$T$. 
\end{lemma}
\begin{proof}
Consider any graph $G=(V,E)$ with $H(G) \le k$. By the definition of the hopping number, there exists a set $B_1 \subseteq V$ of cardinality $k$ such that one can initially colour vertices of $B_1$ blue, and then turn everything blue after a sequence of hops. 

Recall that at time step $t \ge 1$, $B_t$ denotes the set of blue vertices (either active, dormant or extinct) and $W_t$ denotes the set of white vertices. Let $B'_t \subseteq B_t$ be the set of extinct blue vertices at time step $t$, that is, the set of blue vertices that already performed a force. Note that initially $B'_1 = \emptyset$ and $W_1 = V \setminus B_1$. More importantly, each hop increases the cardinality of $B'_t$ by one (one active blue vertex becomes extinct) and decreases the cardinality of $W_t$ by one (one white vertex becomes blue, either active or dormant but certainly not extinct). In particular, at time $t=(n-k)/2$, $|B'_{(n-k)/2}| = |W_{(n-k)/2}| = (n-k)/2$. On the other hand, $B_t \setminus B'_t$ keeps changing during the process but its cardinality is equal to $k$ for all $t$. 

Another important property is that for any $t$, there is no edge between vertices in $B'_t$ and vertices in $W_t$. Indeed, when an active vertex $v \in A_t \subseteq B_t$ performs a force at time $t$, all of its neighbours have to be blue (that is, none of them is in $W_t$). This vertex is moved to $B'_{t+1}$ and, since $W_t$ is only shrinking (that is, $W_1 \supset W_2 \supset \ldots$), the desired property will be preserved to the end of the process. 

The lemma follows by taking the following partition of the vertex set $V$: $S = B'_{(n-k)/2}$, $T = W_{(n-k)/2}$, and $U = B_{(n-k)/2} \setminus B'_{(n-k)/2} = V \setminus (S \cup T)$.
\end{proof}

The above lemma informally tells us that if the hopping number of some graph $G$ is small, then there are two large sets $S$ and $T$ without any edges in between. Such situation does not happen in good expanders, in particular, this property is not satisfied a.a.s.\ for dense random $d$-regular graphs. It provides a lower bound for the hopping number of $\mathcal{G}_{n,d}$ that holds a.a.s. We provide two arguments yielding two corresponding lower bounds. The second bound is stronger for $d \ge 35$.

\medskip

For convenience, we repeat the statement of Theorem~\ref{upperBoundExpanderMixing}.
Let $G=(V,E)$ be a $d$-regular graph with $n$ vertices and set $\lambda = \lambda(G)$. Then,
$$
H(G) \ge \max \left( 1 - \frac {2\lambda}{d}, \frac {d-\lambda}{d+3\lambda} \right) n ~~=~~ \left( 1 - \min \left( \frac {2 \lambda}{d}, \frac {4\lambda}{d+3\lambda} \right) \right) n.
$$ 
As a result, for any $d \ge 3$ and $\eps > 0$, a.a.s. 
$$
H(\mathcal{G}_{n,d}) \ge \left( 1 - \min \left( \frac {4 \sqrt{d-1}}{d}, \frac {8\sqrt{d-1}}{d+6\sqrt{d-1}} \right) -\eps \right) n.
$$ 
\begin{proof}[Proof of Theorem~\ref{upperBoundExpanderMixing}]
Consider any $d$-regular graph $G$ on $n$ vertices and set $\lambda = \lambda(G)$. It follows immediately from the Expander Mixing Lemma (Lemma~\ref{lem:AC}) that for any two disjoint sets $S$ and $T$ of cardinality $(n-k)/2$,
$$
|E(S,T)| \ge \frac {d|S||T|}{n} - \lambda \sqrt{|S||T|} = \frac {d(n-k)^2}{4n} - \lambda \frac {n-k}{2} = \frac {d(n-k)}{4} \left( \frac {n-k}{n} - \frac {2\lambda}{d} \right) > 0,
$$
as long as $k < (1-2\lambda/d) n$. Hence, by Lemma~\ref{lem:big_sets}, we get that 
$$
H(G) \ge \left( 1 - \frac {2\lambda}{d} \right) n.
$$

To get the second lower bound, let us again consider any two disjoint sets $S$ and $T$ of cardinality $(n-k)/2$. It follows from~(\ref{eqn:bisection}) that
$$
|E(S,V \setminus S)| \ge (d-\lambda) \frac{|S||V \setminus S|}{n} = (d-\lambda) \frac{(n-k)(n+k)}{4n},
$$
and the same lower bound holds for $|E(T,V \setminus T)|$. Similarly, using~(\ref{eqn:bisection}) one more time, after setting $U = V \setminus (S \cup T)$, we get that
$$
|E(U,V \setminus U)| \le (d+\lambda) \frac{|U||V \setminus U|}{n} = (d+\lambda) \frac{k(n-k)}{n},
$$
since $|U| = n - (n-k)/2 - (n-k)/2 = k$. We get that 
\begin{eqnarray*}
|E(S,T)| &=& \frac 12 \Big( |E(S,V\setminus S)| + |E(T,V\setminus T)| - |E(U,V\setminus U)| \Big) \\
&\ge& \frac 12 \Big( (d-\lambda) \frac{(n-k)(n+k)}{4n} + (d-\lambda) \frac{(n-k)(n+k)}{4n} - (d+\lambda) \frac{k(n-k)}{n} \Big) \\
&=& \frac {n-k}{4n} \Big( (d-\lambda) (n+k) - 2(d+\lambda) k \Big) \\
&=& \frac {n-k}{4n} \Big( (d-\lambda) n - (d+3\lambda) k \Big) ~~>~~ 0,
\end{eqnarray*}
provided that $k < \frac {d-\lambda}{d+3\lambda} n$, and so by Lemma~\ref{lem:big_sets} we get that
$$
H(G) \ge  \frac {d-\lambda}{d+3\lambda} \, n.
$$

The conclusion for $\mathcal{G}_{n,d}$ follows immediately from Lemma~\ref{lem:Fri} which finishes the proof of the theorem.
\end{proof}

	\section{Lower Bounds From the Configuration Model: $d$-regular Graphs, $d \ge 3$}\label{sec:lower_bounds_config}

We will continue exploiting Lemma~\ref{lem:big_sets} to get stronger (but numerical, not explicit) lower bounds for the hopping number. This time we will use the configuration model (see Subsection~\ref{sec:config}) to show that there are no two large disjoint sets in $\mathcal{G}_{n,d}$ with no edge between them. 

\medskip

For a given integer $d \ge 3$, let 
\begin{eqnarray}
g_d(x,z) &=& \left( \frac {d}{2} -1 - dz \right) x \log(x) + (d-1) (1-x) \log \left( \frac {1-x}{2} \right) \nonumber \\
&& - \, 2dxz\log(z) \nonumber \\
&& - \, \frac {(1-2z)dx}{2} \log(1-2z) \nonumber \\
&& - \, d \left( \frac{1-x}{2}-zx \right) \log \left( \frac{1-x}{2}-zx \right). \label{eq:gd_function}
\end{eqnarray}
For a fixed $x \in (0,1)$, function $g_d(x,z)$ is maximized at
\begin{equation}
z_0(x) := \frac {1 - \sqrt{ 1-2(1-x)x }}{2x}\,. \label{eq:z0}
\end{equation}
(Note that $z_0(x)$ does not depend on $d$; see the proof of Theorem~\ref{thm:configuration} for more details.) Let $h_d(x) = g_d(x,z_0(x))$ be the corresponding maximum value. (For an illustration, we plot functions $h_3(x)$ and $h_{10}(x)$ in Figure~\ref{fig:function_h_d3_d10}.)
\begin{figure}[h]
\centering
\includegraphics[width=0.6\textwidth]{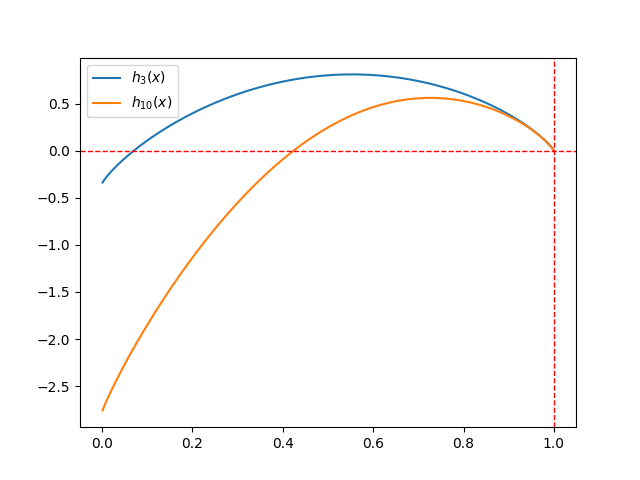}
\caption{Function $h_3(x)$ and $h_{10}(x)$.}\label{fig:function_h_d3_d10}
\end{figure}

\noindent 
Function $h_d(x)$ has the following properties: $\lim_{x \to 0} h_d(x) = - (d-2) \log(2) / 2 < 0$, it is increasing on the interval $(0, \hat{x}]$ (for some $\hat{x} \in (0,1)$), decreasing on the interval $[\hat{x}, 1)$, and $\lim_{x \to 1} h_d(x) = 0$. As a result, there is a unique value $x_d \in (0, 1)$ for which $h_d( x_d ) = 0$. This value can be easily approximated numerically and will play a crucial role in our next result. In light of the above definitions, the following theorem is equivalent to Theorem~\ref{upperBoundConfig}.

\begin{theorem}\label{thm:configuration}
Fix any integer $d \ge 3$ and $\eps > 0$. Let $x_d$ be the unique $x \in (0,1)$ for which $h_d(x)=g_d(x,z_0(x))=0$, where $g_d(x,z)$ and $z_0(x)$ are defined in~(\ref{eq:gd_function}) and in~(\ref{eq:z0}), respectively. Then, a.a.s.\ 
$$
H(\mathcal{G}_{n,d}) > (x_d - \eps)n. 
$$
\end{theorem}
\begin{proof}
Fix $d \ge 3$ and consider the configuration model generating $\mathcal{G}_{n,d}$. Suppose that for some carefully chosen function $x=x(n)$ ($0 < x < 1$), the expected number $S(x)$ of partitions of $n$ vertices of $\mathcal{G}_{n,d}$ into sets $S$, $T$, and $U$ such that $|S|=|T|=(1-x)n/2$ , $|U|=xn$, with the property that there is no edge between $S$ and $T$ is $o(1)$. Then, it follows from the first moment method that a.a.s.\ there is no such partition in $\mathcal{G}_{n,d}$ and we immediately get from Lemma~\ref{lem:big_sets} that a.a.s.\ $H(\mathcal{G}_{n,d}) > xn$.

Let us fix some auxiliary functions $y=y(n)$ and $z=z(n)$ such that $y > 0$, $z > 0$, $y+z < 1$, $yx < (1-x)/2$, and $zx < (1-x)/2$. Let $S(x,y,z)$ be the expected number of partitions into sets $S$, $T$, and $U$ such that $|S|=|T|=(1-x)n/2$ , $|U|=xn$, with the properties that $|E(U,S)| = ydxn$, $|E(U,T)| = zdxn$, and $|E(S,T)|=0$. It is clear that
\begin{eqnarray*}
S(x,y,z) &=& \binom{n}{xn} \binom{(1-x)n}{(1-x)n/2} \binom{dxn}{ydxn} \binom{d(1-x)n/2}{ydxn} (ydxn)! \\
&& \cdot \binom{(1-y)dxn}{zdxn} \binom{d(1-x)n/2}{zdxn} (zdxn)! \\
&& \cdot M((1-y-z)dxn) M(d((1-x)/2-yx)n) M(d((1-x)/2-zx)n) / M(dn),
\end{eqnarray*}
where $M(i)$ is the number of perfect matchings on $i$ points, that is,
$$
M(i) = \frac {i!}{(i/2)! 2^{i/2}}.
$$
Indeed, we first need to select vertices to form set $U$ (term $\binom{n}{xn}$) and partition the remaining vertices into $S$ and $T$ (term $\binom{(1-x)n}{(1-x)n/2}$). After that we need to select points in $U$ (term $\binom{dxn}{ydxn}$) and points in $S$ (term $\binom{d(1-x)n/2}{ydxn}$) that will be matched in the configuration model with edges between $U$ and $S$; there are $(ydxn)!$ ways to do that. Then, we need to select points from the remaining points in $U$ (term $\binom{(1-y)dxn}{zdxn} $) and points in $T$ (term $\binom{d(1-x)n/2}{zdxn}$) and match them (term $(zdxn)!$) to form edges between $U$ and $T$. Finally, we independently and arbitrarily pair the remaining points in $U$ (term $M((1-y-z)dxn)$), in $S$ (term $M(d((1-x)/2-yx)n)$), and in $T$ (term $M(d((1-x)/2-zx)n)$). Since each of these configurations occurs with the same probability, namely, with probability $1/M(dn)$, we get the expected value by dividing the product of all above terms by $M(dn)$.

After simplification we get
\begin{eqnarray*}
S(x,y,z) &=& n! (dxn)! (d(1-x)n/2)!^2 2^{dn/2} (dn/2)! \\
&& \cdot (xn)!^{-1} ((1-x)n/2)!^{-2} (ydxn)!^{-1} (zdxn)!^{-1} \\
&& \cdot 2^{-(1-y-z)dxn/2} ((1-y-z)dxn/2)!^{-1} \\
&& \cdot 2^{-d((1-x)/2-yx)n/2} (d((1-x)/2-yx)n/2)!^{-1} \\
&& \cdot 2^{-d((1-x)/2-zx)n/2} (d((1-x)/2-zx)n/2)!^{-1} (dn)!^{-1}.
\end{eqnarray*}
Using Stirling's formula ($i! \sim \sqrt{2\pi i} (i/e)^i$) we obtain
$$
S(x,y,z) = \Theta(n^{-2}) \exp \Big( f_d(x,y,z) n \Big), 
$$
where
\begin{eqnarray*}
f_d(x,y,z) &=& \left( d-1 - d(y+z) - \frac {(1-y-z)d}{2} \right) x \log(x) \\
&& + (d-1) (1-x) \log \left( \frac {1-x}{2} \right) \\
&& - dxy\log(y) - dxz\log(z) \\
&& - \frac {(1-y-z)dx}{2} \log(1-y-z) \\
&& - \frac {d}{2} \left( \frac{1-x}{2}-yx \right) \log \left( \frac{1-x}{2}-yx \right) 
- \frac {d}{2} \left( \frac{1-x}{2}-zx \right) \log \left( \frac{1-x}{2}-zx \right).
\end{eqnarray*}

Not surprisingly, for a fixed $x \in (0,1)$, the function $f_d(x,y,z)$ is maximized when $z=y$. To see it, for example, we may compute the directional derivative of $f_d(x,y,z)$ in the direction $(0,1,-1)$, which is equal to 
$$
\frac {dx}{2} \left( 2 \Big( \log(z) - \log(y) \Big) + \log \left( \frac {1-x}{2} -xy \right) - \log \left( \frac {1-x}{2} -xz \right) \right).
$$
Clearly, if $y<z$, then this derivative is positive which implies that the maximum is obtained when $z=y$. This motivates the definition~(\ref{eq:gd_function}) of $g_d(x,z)$, since $g_d(x,z) = f_d(x,z,z)$.  To maximize $g_d(x,z)$, we compute the derivative with respect to $z$:
$$
\frac {\partial g(x,z)}{ \partial z} = dx \log \left( \frac {(1-x-2zx)(1-2z)}{2xz^2} \right). 
$$
It follows that $\frac {\partial g(x,z)}{ \partial z} = 0$ if and only if $(1-x-2zx)(1-2z) = 2xz^2$. By solving this quadratic equation we get that 
$$
z = \frac {1/x \pm \sqrt{ \Delta }}{2}, \qquad \text{where} \qquad \Delta = \frac {1}{x^2} - \frac {2}{x} + 2.
$$
The constraint $y+z<1$ implies that $z < 1/2$ and, since the larger root is at least $1/(2x) > 1/2$, it is not a feasible solution. This motivates the definition of~(\ref{eq:z0}), as $z_0(x) = (1/x-\sqrt{\Delta})/2$.

We conclude that for any $0 \le y,z \le 1$ satisfying our constraints, 
$$
f_d( x, y, z) \le h_d(x) = g_d(x, z_0(x)) = f_d( x, z_0(x), z_0(x) ). 
$$
Finally, recall that $x_d \in (0, 1)$ is the unique value of $x$ for which $h_d( x ) = 0$. Moreover, $h_d( x_d - \eps ) < 0$. As a result, the expected number $S(x_d-\eps)$ of partitions of the vertex set into sets $S$, $T$, and $U$ such that $|S|=|T|=(1-x_d+\eps)n/2$ , $|U|=(x_d-\eps)n$, with the property that there is no edge between $S$ and $T$ can be estimated as follows
\begin{eqnarray*}
S(x_d - \eps) &=& \sum_{y,z} S(x_d-\eps,y,z) = \sum_{y,z} \bigo( n^{-2}) \exp \Big( f_d(x_d - \eps,y,z) n \Big) \\
&=& \sum_{y,z} \bigo( n^{-2}) \exp \Big( h_d(x_d - \eps) n \Big) = \bigo( 1 ) \exp \Big( h_d(x_d - \eps) n \Big) \\
&=& \bigo( 1 ) \exp \Big( - \Omega ( n ) \Big) = o(1).
\end{eqnarray*}
Hence, a.a.s.\ there is no partition with this property and Lemma~\ref{lem:big_sets} implies that a.a.s.\ $H(\mathcal{G}_{n,d}) > (x_d - \eps)n$.  This finishes the proof of the theorem. 
\end{proof}

\section{Upper Bound from the Degree-greedy Algorithm: $3$-regular graphs} \label{sec:DEattempt}

In this section, we assume that $d=3$ is fixed with $dn$ even. In order to get an asymptotically almost sure upper bound on the hopping number, we study an algorithm that selects random vertices of minimum degree and tries to hop from them. This algorithm is called \emph{degree-greedy} because the vertex is chosen from those with the lowest degree. Similar technique was successfully used to estimate the brush number of random $d$-regular graphs~\cite{alon2009cleaning}. We are not as successful in our present problem but we do hope that one can modify our algorithm, and use similar techniques to analyze it, to get better bounds than the ones we managed to prove. We illustrate our ideas in the simplest case, namely, for $d=3$.

We start with a random $3$-regular graph $G=(V,E)$ on $n$ vertices, and we will work with the configuration model. During the process, we will keep track of a set $D_t$ of vertices that have at least one point that is still unmatched. Vertices in $D_t$ will be considered as potential candidates to hop from. The process will ensure that these vertices have not perform a force yet. Moreover, white vertices are easy to identify, namely, they have 3 unmatched points. 

Initially, $D_0=V$. In every step $t$ of the process, we select a random vertex $\alpha_t$, chosen uniformly at random from those vertices with the lowest degree in the induced subgraph $G[D_{t-1}]$ on unexposed points. We expose the neighbours associated with the remaining points of $\alpha_t$, make them blue (if they are still white), and try to hop through one of them. To be able to do it, one of these neighbours has to have at least one remaining unchosen point which is associated with a white vertex. Regardless, whether we succeed to hop or not, $\alpha_t$ has to be removed from $D_t$ together with other vertices that got all points exposed. 

In the first step, a vertex of degree 3 is selected to hop from.  Three of its neighbours become blue and we hop through one of them, making another vertex blue. (A.a.s.\ there is no triangle around the initial vertex.) The induced subgraph $G[D_1]$ now has $1$ vertex of degree $1$, $3$ vertices of degree $2$, and $n-5$ vertices of degree $3$. Note that $\alpha_1$ is a.a.s.\ the only vertex whose degree in $G[D_t]$ is $3$ at the time of selecting a vertex to hop from. Indeed, if $\alpha_t$ ($t \ge 2$) has degree $3$ in $G[D_{t-1}]$, then $G[D_{t-1}]$ consists of some connected components of $G$ and thus $G$ is disconnected. It was proven in~\cite{wormald1981asymptotic} that for constant $d$, $G$ is disconnected with probability $o(1)$ (this also holds when $d$ is growing with $n$, as shown in~\cite{luczak1992sparse}). 

In the second step, we try hop from the vertex of degree $1$. A.a.s.\ its neighbour is white and after hopping through it, this vertex becomes of degree $1$. However, when vertices of degree $2$ become plentiful, we will be hopping through them often (making them of degree $0$) and we might run out of vertices of degree $1$. We will start hopping from vertices of degree $2$ but eventually come back to hopping from vertices of degree $1$. Our goal is to control such ``hiccups''. The details of the application of the differential equations method to such degree-greedy algorithms have been omitted, but can be found in~\cite{wormald2003analysis}.

\medskip

For $0\le i\le 3$, let $Y_i=Y_i(t)$ denote the number of vertices with $i$ unmatched points at time~$t$. (Note that $Y_0(t)=n-\sum_{i=1}^3 Y_i(t)$ so $Y_0(t)$ does not need to be calculated, but it is useful in the discussion.) Let $S(t)=\sum_{i=1}^3 i Y_i(t)$ be the total number of points that are unmatched at time $t$. It is tedious but not so difficult to consider all the cases to see that for $r \in \{1, 2\}$ and $i \in \{0, 1, 2, 3\}$,
\begin{align}
\E \Big( Y_i(t+1)-Y_i(t) & ~|~ G[D_{t}] ~\wedge~ \deg_{G[D_{t}]}(\alpha_{t+1}) = r \Big) \nonumber \\
& = f_{i,r} (t/n, Y_1(t)/n, Y_2(t)/n, Y_3(t)/n) \nonumber \\
& = \sum_{j} p_r^{(j)} \Delta Y_i^{(j)}, 
\label{eq:expected_change}
\end{align}
where $p_r^{(j)}$ and $\Delta Y_i^{(j)}$ are presented in Tables~\ref{table_r1} and~\ref{table_r2} (the sum is over all possible cases: $j \in [5]$ for $r=1$, and $j \in [12]$ for $r=2$). Indeed, let us explain, for example, the very first case, $r=1$ and $j=1$. We try to hop from a vertex of degree 1. Its neighbour had degree~3 with probability $3Y_3(t) / (S(t)-1) \sim 3Y_3(t) / S(t)$ but will have degree 1 once we hop through it. Its neighbour is white with probability $3(Y_3(t)-1) / (S(t)-3) \sim 3Y_3(t) / S(t)$. Hence, this case happens with probability asymptotic to $p_1^{(1)} = (3Y_3(t)/S(t))^2$. We can hop there, making this effort count. Three vertices changed their degrees. Initially we have one vertex of degree 1 and 2 of degree 3. After the hop, one of them became of degree 0, one of degree 1, and one of degree 2, explaining the $\Delta Y_i^{(1)}$'s. Note that sometimes we are not able to make a hop, as none of the vertices in the second neighbourhood were initially white. We indicate whether the hop occurred or not in the last column in the two corresponding tables.

{\scriptsize
\begin{table}[htp]
\caption{Vertex of degree $r=1$ tries to hop ($5$ cases)}
\begin{center}
\begin{tabular}{|c|c|c|c|c|c|c|}
\hline
$j$ & $p_1^{(j)}$ & $\Delta Y_0^{(j)}$ & $\Delta Y_1^{(j)}$ & $\Delta Y_2^{(j)}$ & $\Delta Y_3^{(j)}$ & succesful hop? \\
\hline
1 & $\left( \frac{3Y_3(t)}{S(t)} \right)^2$ & $1$ & $0$ & $1$ & $-2$ & Yes \\
\hline
2 & $\left( \frac{3Y_3(t)}{S(t)} \right)^2 \left( \frac{2Y_2(t)}{S(t)}\right)$ & $2$ & $0$ & $0$ & $-2$ & Yes \\
\hline
3 & $\left( \frac{3Y_3(t)}{S(t)} \right) \left( \frac{2Y_2(t)}{S(t)}\right)^2$ & $2$ & $1$ & $-2$ & $-1$ & No \\
\hline
4 & $\left( \frac{3Y_3(t)}{S(t)} \right) \left( \frac{2Y_2(t)}{S(t)}\right)$ & $2$ & $-1$ & $0$ & $-1$ & Yes \\
\hline
5 & $\left( \frac{2Y_2(t)}{S(t)} \right)^2$ & $2$ & $0$ & $-2$ & $0$ & No \\
\hline
\end{tabular}
\end{center}
\label{table_r1}
\end{table}%
}

{\scriptsize
\begin{table}[htp]
\caption{Vertex of degree $r=2$ tries to hop ($12$ cases)}
\begin{center}
\begin{tabular}{|c|c|c|c|c|c|c|}
\hline
$j$ & $p_2^{(j)}$ & $\Delta Y_0^{(j)}$ & $\Delta Y_1^{(j)}$ & $\Delta Y_2^{(j)}$ & $\Delta Y_3^{(j)}$ & succesful hop? \\
\hline
1 & $\left( \frac{3Y_3(t)}{S(t)} \right)^2 \left( \frac{2Y_2(t)}{S(t)}\right)$ & $1$ & $1$ & $1$ & $-3$ & Yes \\
\hline
2 & $\left( \frac{3Y_3(t)}{S(t)} \right)^3 \left( \frac{2Y_2(t)}{S(t)}\right)$ & $2$ & $1$ & $0$ & $-3$ & Yes \\
\hline
3 & $\left( \frac{3Y_3(t)}{S(t)} \right)^3 \left( \frac{2Y_2(t)}{S(t)}\right)^2$ & $2$ & $3$ & $-2$ & $-3$ & Yes \\
\hline
4 & $\left( \frac{3Y_3(t)}{S(t)} \right)^3 \left( \frac{2Y_2(t)}{S(t)}\right)^3$ & $3$ & $3$ & $-3$ & $-3$ & Yes \\
\hline
5 & $\left( \frac{3Y_3(t)}{S(t)} \right)^2 \left( \frac{2Y_2(t)}{S(t)}\right)^4$ & $3$ & $4$ & $-5$ & $-2$ & No \\
\hline
6 & $2 \left( \frac{3Y_3(t)}{S(t)} \right)^2 \left( \frac{2Y_2(t)}{S(t)}\right)$ & $2$ & $0$ & $0$ & $-2$ & Yes \\
\hline
7 & $2 \left( \frac{3Y_3(t)}{S(t)} \right)^2 \left( \frac{2Y_2(t)}{S(t)}\right)^2$ & $2$ & $2$ & $-2$ & $-2$ & Yes \\
\hline
8 & $2 \left( \frac{3Y_3(t)}{S(t)} \right)^2 \left( \frac{2Y_2(t)}{S(t)}\right)^3$ & $3$ & $2$ & $-3$ & $-2$ & Yes \\
\hline
9 & $2 \left( \frac{3Y_3(t)}{S(t)} \right) \left( \frac{2Y_2(t)}{S(t)}\right)^4$ & $3$ & $3$ & $-5$ & $-1$ & No \\
\hline
10 & $\left( \frac{3Y_3(t)}{S(t)} \right) \left( \frac{2Y_2(t)}{S(t)}\right)^2$ & $2$ & $1$ & $-2$ & $-1$ & Yes \\
\hline
11 & $\left( \frac{3Y_3(t)}{S(t)} \right) \left( \frac{2Y_2(t)}{S(t)}\right)^3$ & $3$ & $1$ & $-3$ & $-1$ & Yes \\
\hline
12 & $\left( \frac{2Y_2(t)}{S(t)}\right)^4$ & $3$ & $2$ & $-5$ & $0$ & No \\
\hline
\end{tabular}
\end{center}
\label{table_r2}
\end{table}%
}

Suppose that at some step $t$ of the process, hopping from a vertex of degree $2$ creates, in expectation, $\beta$ vertices of degree
$1$ and hopping from a vertex of degree $1$ decreases, in expectation, the number of vertices of degree $1$ by $\tau$. After hopping from a vertex of degree $2$, we expect to then hop (on average) from $\beta/\tau$ vertices of degree $1$. Thus, the proportion of steps which hop from vertices of degree $2$ is $1/(1+\beta/\tau) = \tau/(\beta+\tau)$. If $\tau$ falls below zero, vertices of degree $1$ begin to build up and do not decrease under repeated hopings of this type and we move to the next phase.

From~(\ref{eq:expected_change}) it follows that
\begin{eqnarray*}
\beta &=& \beta(x,y_1,y_2, y_3) ~~=~~ f_{1,2}(x,y_1,y_2, y_3) ~~=~~ f_{1,2} (x,\textbf{y}),\\
\tau &=& \tau(x,y_1,y_2, y_3) ~~=~~ -f_{1,1}(x,y_1,y_2, y_3) ~~=~~ -f_{1,1} (x,\textbf{y}),
\end{eqnarray*}
where $x=t/n$ and $y_i(x)=Y_i(t)/n$ for $i \in [3]$. This suggests the following system of differential equations
$$
\frac {d y_i}{dx} = F(x,\textbf{y},i)
$$
where
\begin{equation*}
F(x,\textbf{y},i) = \frac{\tau}{\beta+\tau} f_{i,2}(x,\textbf{y}) + \frac {\beta}{\beta+\tau} f_{i,1}(x,\textbf{y}).
\end{equation*}
At this point we may formally define the termination point $\hat{x}$ is defined as the infimum of those $x>\hat{x}$ for which at least one of the following holds: $\tau \le 0$, $\tau+\beta=0$, or $y_{2} \le 0$. The initial conditions are $y_3(0)=1$ and $y_i(0)=0$ for $i \in \{0, 1, 2\}$.

The general result~\cite[Theorem~1]{wormald2003analysis} studies a deprioritized version of degree-greedy algorithms, which means that the vertices are chosen  to process in a slightly different way, not always the minimum degree, but usually a random mixture of two degrees. Once a vertex is chosen, it is treated the same as in the degree-greedy algorithm. The variables $Y$ are defined in an analogous manner. The hypotheses of the theorem are mainly straightforward to verify but require several inequalities involving derivatives to hold at the termination of phases, for the full rigorous conclusion to be obtained. However, in practice, the equations are simply solved numerically in order to find the points $\hat{x}$, since a fully rigorous bound is not obtained unless one obtains strict inequalities on the values of the solutions. The conclusion is that, for a certain algorithm using a deprioritized ``mixture'' of the steps of the degree-greedy algorithm, with variables $Y_i$ defined as above, we have that a.a.s.
$$
Y_i(t) = n y_i(t/n)+o(n)
$$
for $0 \le i\le 3$. We omit all details, pointing the reader to~\cite{wormald2003analysis} and Subsection~\ref{sec:DEmethod} for the differential equations method which is a main tool in proving the result.  In addition, the theorem gives information on an auxiliary variable such as, of importance to our present application, the number of vertices that actually hoped. 

\medskip

The numerical solution to the relevant differential equations is shown in Figure~\ref{fig:DEMethodGraph}.  During the process, only vertices of degree 1 and 2 attempt to make a hop. Since we prioritize vertices of degree~1, the function $y_1$ responsible for ``monitoring'' such vertices is equal to zero. Note that it means that a.a.s.\ there are only $o(n)$ vertices of degree 1 at any point of the process, not that we do not see them at all (we clearly do). As expected, the number of vertices of degree 3 (represented by $y_3$) decreases, and the number of vertices of degree 0 (represented by $y_0$) increases. The number of vertices of degree 2 (represented by $y_2$) initially increases but then it decreases. A.a.s., at time $\hat{x}n \approx 0.6614n$ the process stops with all vertices being of degree 0. By investigating the auxiliary random variable, and the associated differential equation, we conclude that a.a.s.\ approximately $0.5159n$ vertices hop, yielding an upper bound for the hoping number of approximately $(1-0.5159)n = 0.4841n$. Unfortunately, it is a much weaker upper bound than the one we obtained by investigating the contiguous model, namely, $(1+o(1)) n/3$.

\begin{figure}
        \centering
       \includegraphics[width=0.6\textwidth]{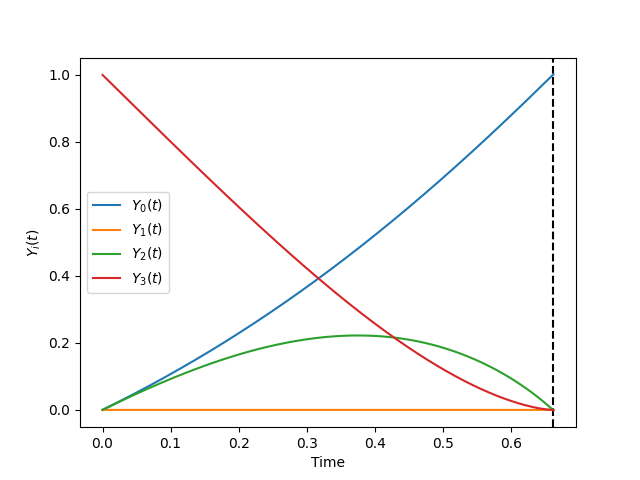}
        \caption{Evolution of $Y_i(t)$, $0\leq i\leq3$ for $d=3$ using the Differential Equation Method.}
        \label{fig:DEMethodGraph}
\end{figure}

\bibliographystyle{plain}
	
	\bibliography{refs.bib}
	
\end{document}